\documentclass{amsart}
\usepackage{amsmath, amsthm, amssymb}
\usepackage{verbatim,url,xcolor}
\usepackage{url}

\renewcommand{\d}{\partial}

\newcommand{\ddbar}{\sqrt{-1}\d\overline{\d}}

\newtheorem{thm}{Theorem}

\newtheorem{lem}[thm]{Lemma}

\theoremstyle{definition}

\renewcommand{\[}{\begin{equation}}
\renewcommand{\]}{\end{equation}}

\newcommand{\la}{\lambda}

\newcommand{\vep}{\varepsilon}
\newcommand{\tomega}{\tilde \omega}

\newcommand{\NN}{\mathbb{N}}

\newcommand{\QQ}{\mathbb{Q}}

\newcommand{\CC}{\mathbb{C}}

\newcommand{\PP}{\mathbb{CP}}

\newcommand{\Rc}{\mathrm{Ric}}
\newcommand{\Vol}{\mathrm{Vol}}
\newcommand{\sR}{\mathcal{R}}
\newcommand{\sS}{\mathcal{S}}

\title[Metric rigidity and almost maximal volume of K\"ahler manifolds]{Metric rigidity of K\"ahler manifolds with lower Ricci bounds and almost maximal volume}
\author[V. V. Datar]{Ved Datar}
\author[H. Seshadri]{Harish Seshadri}
\author[J. Song]{Jian Song}

\address{Department of Mathematics, Indian Institute of Science, Bangalore, India - 560012}
\email{vvdatar@iisc.ac.in, harish@iisc.ac.in}

\address{ Department of Mathematics,  Rutgers University, Piscataway, NJ 08854}
\email{jiansong@math.rutgers.edu}

\thanks{Research supported in
part by National Science Foundation grant DMS-1711439, UGC (Govt.\   of India) grant no.\ F.510/25/CAS- II/2018(SAP-I), and the Infosys Young investigator award.}

\begin{document}
\begin{abstract}
In this short note we prove that a K\"ahler manifold with lower Ricci  curvature bound and almost maximal volume  is Gromov-Hausdorff close to the projective space with the Fubini-Study metric. This is done by combining the recent results on holomorphic rigidity of such K\"ahler manifolds \cite{Z1} with the structure theorem of Tian-Wang \cite{TW} for almost Einstein manifolds. This can be regarded as the  complex analog of the result  on Colding on the shape of Riemannian manifolds with  almost maximal volume.
\end{abstract}
\maketitle

\section{Introduction}
In this note we wish to study metric rigidity of K\"ahler manifolds $(M^n,\omega)$ satisfying 
\begin{align}\label{eq:ric}
\Rc(\omega) \geq  \omega,
\end{align} 
and with almost maximal volume. Recently Zhang \cite{Z1} proved that any K\"ahler manifold satisfying \eqref{eq:ric} must have $$\Vol(M,\omega):= \int_{M}\omega^n \leq \Vol(\PP^n,\omega_{\PP^n}).$$ Here $\omega_{\PP^n}$ is the Fubini-Study metric on $\PP^n$ with $\Rc(\omega_{\PP^n}) = \omega_{\PP^n}$. Moreover, Zhang proved that the maximal volume is attained if and only if $(M,\omega)$ is isometric to $(\PP^n,\omega_{\PP^n})$. For K\"ahler-Einstein Fano manifolds such optimal bounds were proved earlier by Berman-Berndtsson \cite{BB} in presence of a $\CC^*$ action with finitely  many fixed  points, and unconditionally by Fujita \cite{F}. On the other hand Colding in \cite{Col} proved that an $n$-dimensional Riemannian manifold with Ricci curvature greater or equal to $(n-1)$   and almost maximal volume  is close to the round sphere in Gromov-Hausdorff distance. The main purpose of this note is to establish the following metric rigidity result as the complex analogue of Colding's theorem. 
\begin{thm}\label{thm:main}
For all $\vep>0$, there exists $\delta = \delta(\vep,n)>0$ such that  if $(M^n, \omega)$ is  a K\"ahler manifold satisfying \eqref{eq:ric} and
$$\Vol(M, \omega)   >  (1-\delta) \Vol(\PP^n, \omega_{\PP^n}),$$
then $$d_{GH}\Big((M,\omega),(\PP^n,\omega_{\PP^n})\Big)  < \vep,$$
where $d_{GH}$ is the Gromov-Hausdorff distance.
\end{thm}

The starting point for this paper is the almost holomorphic rigidity proved by Liu in the appendix of \cite{Z1}. Liu proved that  if $(M^n,\omega)$ is a K\"ahler manifold satisfying \eqref{eq:ric},   $M$ must be biholomorphic to $\PP^n$ if the volume of $(M^n, \omega)$ is sufficiently close to that of $(\PP^n, \omega_{\PP^n})$.  This can be regarded as a complex version of Perelman's result in \cite{P}. In particular, the K\"ahler manifold $M$ in Theorem \ref{thm:main} must be $\PP^n$ and indeed Theorem \ref{thm:main} is an analytic or metric extension of Liu and Zhang's theorem. The proof of Theorem \ref{thm:main} relies on the structure theorem of  Tian and Wang   \cite{TW} on Gromov-Hausdorff limits of almost Einstein manifolds. We also offer an alternate proof relying on the recent  results of Liu and  Szekelyhidi on structure of  non-collapsed Gromov-Hausdorff limits of K\"ahler manifolds with a Ricci curvature lower bound.

The results of \cite{Z1} rely on recent works on stability thresholds and $K$-stability  from algebraic geometry. It will be interesting to obtain an independent differential geometric proof for the holomorphic rigidity of $\PP^n$.

\section{Proof of the main theorem}

We will first prove the following general result.  

\begin{thm}\label{thm:gh-einstein}
Let $(M^n,\omega_{KE})$ be a K\"ahler-Einstein manifold. Let $\delta_i\rightarrow 0$ and $\omega_i\in   c_1(M)$ such that $$\Rc(\omega_i) \geq (1-\delta_i)\omega_i.$$ Then $$(M,\omega_i)\xrightarrow{d_{GH}}(M,\omega_{KE}).$$
\end{thm}

Before we begin the proof, let us recall the definition of almost K\"ahler-Einstein  manifolds from \cite{TW}. A sequence of closed K\"ahler manifolds $(M_i^n,\omega_i,p_i)$ is said to be  {\em almost K\"ahler-Einstein} if the following conditions are satisfied. 

\begin{itemize}
\item $\Rc(\omega_i) \geq -\omega_i$
\item $p_i\in M_i$ and $$|B_{\omega_i}(p_i,1)| \geq \kappa>0.$$ 
\item$$ F_i:= \int_{M_i}|\Rc(\omega_i) - \la_i\omega_i| \omega_i^n \xrightarrow{i\rightarrow\infty} 0.$$
\item For some $\la_i\in [-1,1]$, the flow $$\frac{\partial \omega_i}{\partial t} = -\Rc(\omega_i) + \la_i \omega_i$$ has a solution on $M_i\times [0,1]$. Moreover, $$E_i:= \int_0^1\int_{M_i}|S_{g_i} - n\la_i|\,\omega_{i}(t)^n\,dt \xrightarrow{i\rightarrow \infty}0.$$
\end{itemize}

\begin{thm}[Theorem 2 in \cite{TW}]\label{thm:ae} Let $(M^n_i,\omega_i,p_i)$ be a sequence of pointed almost K\'ahler-Einstein manifolds of (complex)  dimension $n$ with $\la_i = 1$. Let $(Z,d)$ be a subsequential Gromov-Hausdorff limit. Then there exist a regular-singular decomposition $Z = \sR\cup \sS$ such that 
\begin{itemize}
\item $\sR$ is a smooth convex, open K\"ahler manifold with complex structure $J_\infty$ and  K\"ahler form $\omega_\infty$ satisfying $$\Rc(\omega_\infty)  = \omega_{\infty}.$$
\item $\dim \sS \leq 2n-4.$
\end{itemize}

\end{thm}

\begin{proof}[Proof of Theorem \ref{thm:gh-einstein}]
First we need the following observation from \cite{TW}, whose proof we reproduce for the convenience of the reader.
\begin{lem}[Theorem 6.2 in \cite{TW}]
The sequence $(M,\omega_i,p)$ from the statement of Theorem \ref{thm:gh-einstein} forms a sequence of almost-Einstein manifolds with $\la_i = 1$.
\end{lem} 
\begin{proof}
The Ricci lower bound is from the hypothesis, and the volume lower bound follows from Bishop-Gromov inequality, Myers theorem and the fact that the volume of $\omega_i$ is constant.  Moreover, it is  well known through the work of Perelman that the Ricci flow $$\begin{cases}\frac{\partial \omega_{i}(t)}{\partial t} = -\Rc(\omega_i(t)) + \omega_i(t)\\ \omega_i(0)=\omega_i\end{cases}$$ exists for all time. All we need to prove is that $F_i$ and $E_i$ converge to zero.  

Note that since $S_{\omega_{i}} - n\geq -n\delta_i$, by the maximum principle for the scalar curvature under Ricci flow, we have the bound $$S_{\omega_i(t)}  - n\geq -n\delta_i e^t.$$ for all $i$ and for all $t$. On the other hand, since the K\"ahler class remains fixed, $$\int_M (S_{\omega_i(t)}-n)\frac{\omega_i(t)^n}{n!} = 0,$$ and hence 
\begin{equation}\label{eq:scl1}
\int_M |S_{\omega_i(t)} - n|\frac{\omega_i(t)^n}{n!}\leq n\delta_ie^t(c_1(M))^n.
\end{equation} 
Integrating in $t$, we obtain the required decay on $E_i$. Next, since  $\Rc(\omega_i) - (1-\delta_i)\omega_i\geq 0$, we have 
\begin{align*}
F_i&:= \int_M|\Rc(\omega_i) - \omega_i|\omega_i^n \\
&\leq \int_M |\Rc(\omega_i) - (1-\delta_i)\omega_i|\omega_i^n + n\delta_ic_1(M)^n\\
&\leq 10n^{3/2}\delta_ic_1(M)^n\xrightarrow{i\rightarrow \infty}0,
\end{align*}
where  we used \eqref{eq:scl1} at $t=0$ to estimate the first integral.
\end{proof}
After passing to a subsequence,  $$(M,\omega_i,p)\xrightarrow{d_{GH}} (Z,d,p_\infty)$$ where $Z$ has the regular-singular  decomposition as in  Theorem \ref{thm:ae}. From the proof of Theorem \ref{thm:ae} in \cite{TW} it follows that  any tangent cone is a metric cone $C(Y)$ over a link $Y$  with singular set  $\sS_Y$  of real co-dimension  at least four.  Moreover, on  the regular part of $C(Y)$, the cone metric $g_{C(Y)}  = dr^2+  r^2g_Y$ is Ricci flat and it's K\"ahler form is given by $$ \omega_{C(Y)}= \frac{\sqrt{-1}}{2}\partial\bar\partial r^2,$$  where $r$  is the distance from the vertex (cf.\ Proposition 5.2 and Lemma 5.2 in \cite{TW}). The arguments in \cite{DS} now apply and we have the following.

\begin{lem}\label{lem:qfano}
\begin{enumerate}
\item For sufficiently large $k\in \NN$, there is a sequence of embeddings $T_i:M\rightarrow \PP^N$ by sections of $H^0(M,-K_M^k)$ which are orthonormal with respect to the metric induced by $\omega_i$ such that the  flat limit $W = \lim_{i\rightarrow \infty}T_i(M)$ is  a normal $\QQ$-Fano variety.
\item The limiting K\'ahler metric $\omega_\infty$ extends globally to  a weak K\"ahler-Einstein metric on $W$.
\end{enumerate}
\end{lem}

By a weak K\"ahler-Einstein metric we mean that $\omega_\infty = \ddbar\varphi_\infty$ where $e^{-r\varphi_\infty}$ is a continuous hermitian metric on $K_W^{-r}$, and $\varphi_\infty$ satisfies the following Monge-Ampere equation $$(\ddbar\varphi_\infty)^n =  e^{-\varphi_\infty}.$$    

Continuing with our proof of the Theorem \ref{thm:gh-einstein}, since $W$ admits a weak K\"ahler-Einstein metric, the Futaki invariant vanishes identically and $\mathrm{Aut}(W)$ is reductive. Then, by the Luna slice theorem, there is a test configuration of $(M,-K_M^r)$ with $(W,\mathcal{O}_{\PP^N}(1))$ as the central fiber. Since $M$ is $K$-stable (by virtue of admitting a K\"ahler-Einstein metric), this forces $W$ to be biholomorphic to $M$ and $\omega_\infty$ to be a smooth K\"ahler-Einstein metric. But then by the uniqueness  of K\"ahler-Einstein metrics, $\omega_\infty$ is isometric to $\omega_{KE}$, and hence $(M,\omega_i)\xrightarrow{d_{GH}} (M,\omega_{KE})$. 
\end{proof}

We would like to remark that Theorem \ref{thm:gh-einstein} can also be proved by using the result in \cite{LS} and we sketch the proof below. By the assumption of Theorem \ref{thm:gh-einstein}, $(M, \omega_i)$ converges to a metric space $(Z, d)$ after passing to a subsequence since the diameter is uniformly bounded above by volume comparison. The main result of \cite{LS} states that $Z$ is an $n$-dimensional normal projective variety. For sufficiently large $k>0$,  the $L^2$-orthonormal basis $\{\sigma^{(i)}_0, ..., \sigma^{(i)}_{N_k }\}$ of $H^0(M, (-K_M)^k)$ with respect to $\omega_i$ and its induced hermitian metric on $-K_X$ converge to an orthonormal basis of $H^0(Z,(- K_Z)^k)$ thanks to the partial $C^0$-estimate from \cite{LS}.  The basis $\{\sigma^{(i)}_0, ..., \sigma^{(i)}_{N_k }\}$ induces a sequence of Fubini-Study metrics 
$$\theta_i = k^{-1} \ddbar \log \left( | \sigma^{(i)}_0|^2+...+ |\sigma^{(i)}_{N_k }|^2 \right)$$ and $\omega_i = \theta_i + \ddbar \varphi_i$ with $\varphi_i$ uniformly bounded in $L^\infty(M)$. Furthermore, if we let $\Omega_i = \left( | \sigma^{(i)}_0|^2+...+ |\sigma^{(i)}_{N_k }|^2 \right)^{-1/k}$ be the induced volume form on $M$, then $\varphi_i$ satisfies the following complex Monge-Ampere equation
$$(\theta_i + \ddbar \varphi_i)^n= e^{-(1-\delta_i)\varphi_i - \delta_i f_i} \Omega_i,  $$
where $\theta_i + \ddbar f_i\geq 0$ and $\int_M e^{-\delta_i f_i}\Omega_i$ is uniformly bounded for all $i$. 
After letting $i \rightarrow \infty$, the limiting equation is given by 
$$(\theta_\infty + \ddbar \varphi_\infty)^n= e^{-\varphi_\infty + F_\infty } \Omega_\infty, $$
for some  global plurisubharmonic function $F_\infty$  on $Z$. The reader can refer to \cite{RS} for more details (cf. Section 3). This immediately implies that $F_\infty$ is a constant and $\omega_\infty= \theta_\infty+ \ddbar \varphi_\infty$ is a K\"ahler-Einstein metric with bounded local potentials. This replaces the the proof of Lemma \ref{lem:qfano} and Theorem \ref{thm:main} is proved by the same argument as above using $K$-stability to show that $Z$ must be biholomorphic to $M$ and $\omega_\infty$ is the unique K\"ahler-Einstein metric on $Z$ up to an automorphism of $Z$. 

\medskip

Now we are ready to prove Theorem \ref{thm:main}.
\begin{proof}[Proof of Theorem \ref{thm:main}] We argue by contradiction. By choosing $\delta$ small enough, by the appendix of \cite{Z1}, we may assume that $M$ is bi-holomorphic to $\PP^n$. In particular, $[\omega]$ is a multiple of $c_1(M)$. Suppose there exists an $\vep>0$, a sequence $\delta_i\rightarrow 0$ and a sequence of metrics $\omega_i$ on $\PP^n$ such that 
$$
\Rc(\omega_i) \geq  \omega_i,$$
$$\Vol(\PP^n, \omega_i) \geq (1-\delta_i) \Vol(\PP^n, \omega_{\PP^n}),$$
but
\begin{align}\label{eq:contradiction} 
d_{GH}\Big((\PP^n,\omega_i),(\PP^n,\omega_{\PP^n})\Big)  \geq \vep.
\end{align}
Consider the rescaled metrics $\tilde\omega_i =  \frac{\Vol(\PP^n, \omega_{\PP^n})^{1/n}}{\Vol(\PP^n, \omega_i)^{1/n}}\omega_i$. Then, $\Vol(\PP^n, \tilde\omega_i) = \Vol(\PP^n, \omega_{\PP^n}) $, and so $\tomega_i\in c_1(\PP^n)$. Moreover, we also have 
\begin{align*}
 \omega_i \leq \tomega_i \leq \frac{1}{1-\delta_i}\omega_i\\
\Rc(\tomega_i) \geq (1-\delta_i)\tomega_i.
\end{align*}
By Theorem \ref{thm:gh-einstein} above,  $(\PP^n,\tomega_i)\xrightarrow{d_{GH}} (\PP^n,  \omega_{\PP^n})$. Since  $\frac{\Vol(\PP^n, \omega_{\PP^n}) }{\Vol(\PP^n, \omega_i) }$ is almost one, we can make sure that for $i>>1$,  
\begin{align*}
d_{GH}\Big((\PP^n,\tomega_i),(\PP^n, \Vol(\omega_i)^{-1/n}\omega_{\PP^n})\Big)  \leq \frac{\vep \Vol(\PP^n, \omega_{\PP^n})^{1/2n}}{2 \Vol(\omega_i)^{1/2n}}.
\end{align*} This contradicts the inequality in \eqref{eq:contradiction}.
\end{proof}

\end{document}